\documentclass{article}
\usepackage{authblk,amsfonts,amsmath,amssymb,amsthm,dsfont,fancyhdr}
\usepackage{graphicx}
\usepackage{tikz} 
\usepackage{latexsym}
\usepackage{gensymb}
\usepackage{fullpage}
\usepackage{enumerate}
\usepackage{enumitem}
\usepackage{xcolor}
\usepackage{caption}
\usepackage{float}
\usepackage[titletoc]{appendix}
\usetikzlibrary[trees]
\usepackage{array,multirow}
\usepackage{lipsum}
\usepackage{mathrsfs}
\usepackage{indentfirst}
\usepackage[hyphens]{url}
\usepackage[colorlinks=true,, allcolors=blue]{hyperref}

\newcommand\keywords[1]{\textbf{Keywords}: #1}
\newcommand\MSC[1]{\textbf{MSC2020}: #1}

\newtheorem{theorem}{Theorem}[section]

\newtheorem{lemma}[theorem]{Lemma}
\newtheorem{remark}[theorem]{Remark}

\newtheorem{definition}[theorem]{Definition}

\newtheorem{problem}[theorem]{Problem}

\DeclareMathOperator{\rk}{rank}

\usepackage{microtype}
\date{}
\begin{document}
\title{Signed graphs with fixed smallest eigenvalue at least $-3$ and their lattices}
\author[a] {Meng-Yue Cao}
\author[a,b] {Jack H. Koolen}
\author[a] {Jing-Yuan Liu}
\author[c,d] {Qianqian Yang}
%\thanks{Corresponding author}}
\affil[a]{\footnotesize{School of Mathematical Sciences, University of Science and Technology of China, Hefei, 230026, People's Republic of China}}
\affil[b]{\footnotesize{CAS Wu Wen-Tsun Key Laboratory of Mathematics, University of Science and Technology of China, Hefei, 230026, People's Republic of China}}
\affil[c] {\footnotesize{Department of Mathematics, Shanghai University, Shanghai 200444, People's Republic of China}}
\affil[d] {\footnotesize{Newtouch Center for Mathematics of Shanghai University, Shanghai 200444, People's Republic of China}}
\maketitle

	\newcommand\blfootnote[1]{%
		\begingroup
		\renewcommand\thefootnote{}\footnote{#1}%
		\addtocounter{footnote}{-1}%
		\endgroup}
	\blfootnote{E-mail addresses: {\tt caomengyue@ustc.edu.cn} (M.-Y. Cao), {\tt koolen@ustc.edu.cn} (J.H. Koolen), {\tt liujingyuan@mail.ustc.edu.} {cn} (J.-Y. Liu), {\tt  qqyang@shu.edu.cn} (Q. Yang)}
\vspace{-50pt}

\begin{abstract}
 In this paper, we consider connected signed graphs with smallest eigenvalue at least $-3-\varepsilon$ for a small positive constant $\varepsilon$. We prove that if such a signed graph has sufficiently large minimum valency, then its smallest eigenvalue is at least $-3$, and the lattice associated with it, which is generated by squared norm $3$ vectors, is a sublattice of a direct sum of the standard lattice $\mathbb{Z}^n$ and copies of the root lattice $E_8$. Moreover, there exist infinitely many connected signed graphs with smallest eigenvalue at least $-3$ containing it as a proper induced subgraph. Furthermore, we discuss signed graphs with smallest eigenvalue $-3$ arising from rootless irreducible unimodular lattices.
 \end{abstract}\par

\par
\MSC{11H06, 05C50}

\section{Introduction}\label{sec intro}
Throughout this paper, all graphs are finite, undirected, simple and non-empty. \par

A \emph{signed graph} is a pair $(G,\sigma)$, where $G=(V(G),E(G))$ is the \emph{underlying graph} of $(G,\sigma)$ and $\sigma:E(G)\rightarrow\{+,-\}$ is a signing function. The \emph{order} of $(G,\sigma)$ is the order of $G$, and the \emph{valency} of each vertex $x\in V(G)$ is the cardinality of the set $\{y\in V(G)\mid \{x,y\}\in E(G)\}$. If all vertices have the same valency, then $(G,\sigma)$ is called a \emph{regular signed graph}.\par
Let $U\subseteq V(G)$. A \emph{subgraph} of $(G,\sigma)$ on vertex set $U$ is a signed graph $(H,\sigma_H)$ with vertex set $V(H)=U$ such that $E(H)\subseteq\binom{V(H)}{2}\cap E(G)$ and $\sigma_H=\sigma|_{E(H)}$. If $E(H)=\binom{V(H)}{2}\cap E(G)$, then $(H,\sigma_H)$ is called the \emph{induced subgraph} of $(G,\sigma)$ on $U$.\par

The \emph{adjacency matrix} of a signed graph $(G,\sigma)$ is the symmetric matrix $A=A(G,\sigma)$ indexed by $V(G)$ whose $(x,y)$-entry is given by 
\[
A_{xy}=
\begin{cases}
1, & \text{if } \{x,y\}\in E(G)\text{ and }\sigma(\{x,y\})=+,\\[2mm]
-1, & \text{if } \{x,y\}\in E(G)\text{ and }\sigma(\{x,y\})=-,\\[2mm]
0, & \text{otherwise}.
\end{cases}
\]
The \emph{eigenvalues} of $(G,\sigma)$ are the eigenvalues of its adjacency matrix $A$. \par

Let $U\subseteq V(G)$. A signed graph $(G,\tau)$ is said to be obtained from $(G,\sigma)$ by \emph{switching} with respect to $U$ if, for every edge $\{x,y\}\in E(G)$,
\[
\tau(\{x,y\})=
\begin{cases}
\sigma(\{x,y\}), & \text{if } x,y\in U \text{ or } x,y\in V(G)\setminus U,\\[2mm]
-\sigma(\{x,y\}), & \text{otherwise}.
\end{cases}
\]
If $(G,\tau)$ can be obtained from $(G,\sigma)$ by switching with respect to some subset of $V(G)$, then $(G,\tau)$ and $(G,\sigma)$ are \emph{switching equivalent}. The \emph{switching class} of a signed graph $(G,\sigma)$, denoted by $[G,\sigma]$ is the set of all signed graphs that are switching equivalent to $(G,\sigma)$. Note that switching preserves the spectrum. Consequently, all signed graphs in the same switching class have the same spectrum. This paper studies signed graphs with fixed smallest eigenvalue and large minimum valency.\par

In 2021, Gavrilyuk, Munemasa, Sano and Taniguchi \cite{Gavrilyuk2021signed} 
gave the following characterization of connected signed graphs whose smallest eigenvalue lies in $(-2, -1]$.
\begin{theorem}[{\cite[Theorem $1.2$]{Gavrilyuk2021signed}}]\label{GMST result on -2}
Let $\lambda$ be a real number inside the interval $(-2,-1]$. There exists a positive integer $\kappa_1(\lambda)$ such that, if a connected signed graph $(G,\sigma)$ has smallest eigenvalue $\lambda_{\min}(G,\sigma)\geq\lambda$ and minimum valency at least $\kappa_1(\lambda)$, then $(G,\sigma)$ is switching equivalent to a signed graph whose underlying graph is a complete graph with each edge signed $+$, and hence $\lambda_{\min} (G,\sigma)=-1$. 
\end{theorem}

Recently, Koolen, Liu, Yang and Cao \cite{koolen2026a} generalized this result to the half-open interval $(-1-\sqrt{2},-1]$. Before stating the result, we introduce the lattice associated with a signed graph.\par

Let $(G,\sigma)$ be a signed graph with smallest eigenvalue $\lambda$. Define $\operatorname*{Gr}(G,\sigma):=A(G,\sigma)-\lfloor\lambda\rfloor\mathbf{I}$. Since $\operatorname*{Gr}(G,\sigma)$ is positive semidefinite, there exists a matrix \(N\) such that \(\operatorname*{Gr}(G,\sigma)=N^{T}N\). 
The integral lattice generated by the columns of \(N\) is called the \emph{lattice associated with signed graph \((G,\sigma)\)} and is denoted by \(\Lambda(G,\sigma)\). Since any two such factorizations of \(\operatorname*{Gr}(G,\sigma)\) yield isomorphic lattices, the isomorphism class of \(\Lambda(G,\sigma)\) is uniquely determined by \(\operatorname*{Gr}(G,\sigma)\). For further background on lattices, see Section \ref{subsec lattices}.\par
 
In terms of the lattice associated with signed graphs, the theorem can be stated as follows.

\begin{theorem}[{\cite[Theorem $1.6$]{koolen2026a}}]\label{KLYC result on -1-sqrt2}
Let $\lambda$ be a real number inside the interval $(-1-\sqrt{2},-1]$. There exists a positive integer $\kappa_2(\lambda)$ such that, if a connected signed graph $(G,\sigma)$ has smallest eigenvalue $\lambda_{\min}(G,\sigma)\geq \lambda$ and minimum valency at least $\kappa_2(\lambda)$, then $\lambda_{\min}(G,\sigma)\geq-2$ and the lattice $\Lambda(G,\sigma)$ is $A_n$ or $D_n$ for some positive integer $n$.
\end{theorem}

In this paper, we obtain a substantial strengthening of the above results as follows. 

\begin{theorem}\label{main result}
There exist a positive real number $\varepsilon$ and a positive integer $\kappa_3$ such that, if a connected signed graph $(G,\sigma)$ has smallest eigenvalue $\lambda_{\min}(G,\sigma)>-3-\varepsilon$ and minimum valency at least $\kappa_3$, then $\lambda_{\min}(G,\sigma)\geq -3$ and the lattice $\Lambda(G,\sigma)$ is a sublattice of a direct sum of the standard lattice $\mathbb{Z}^n$ and copies of the root lattice $E_8$. 
\end{theorem}

In the following subsection, a refinement of Theorem \ref{main result} is presented for connected signed graphs with smallest eigenvalue at least $-3$.

\subsection{Maximal and non-extendable signed graphs}

This subsection is concerned with maximal and non-extendable signed graphs with smallest eigenvalue at least $-3$. We will explore such signed graphs from the lattice viewpoint.\par

\begin{definition}
Let $(G,\sigma)$ be a connected signed graph with smallest eigenvalue $\lambda$. The signed graph $(G,\sigma)$ is \emph{maximal}, if there is no connected signed graph $(G',\sigma')$ satisfying the following conditions $\rm{(i)}$, $\rm{(ii)}$ and $\rm{(iii)}$, and is \emph{non-extendable}, if there is no connected signed graph $(G',\sigma')$ satisfying the following conditions $\rm{(i)}$ and $\rm{(ii)}$.
\begin{enumerate}[label=\rm{(\roman*)}]
    \item The signed graph $(G,\sigma)$ is a proper induced subgraph of $(G',\sigma')$.
    \item The smallest eigenvalue of $(G',\sigma')$ is at least $\lfloor\lambda\rfloor$.
    \item $\rk(A(G',\sigma')-\lfloor\lambda\rfloor\mathbf{I})=\rk(A(G,\sigma)-\lfloor\lambda\rfloor\mathbf{I})$.
\end{enumerate}
\end{definition}
A connected signed graph is said to be \emph{extendable}, if it is not non-extendable. By definition, whether a signed graph is maximal or non-extendable depends only on its switching class. In particular, if \((G,\sigma)\) is maximal (resp. non-extendable), then every signed graph in \([G,\sigma]\) is maximal (resp. non-extendable).\par

For a connected signed graph $(G,\sigma)$ with smallest eigenvalue $\lambda$, if $(G,\sigma)$ is extendable, then it is a proper induced subgraph of a connected signed graph $(G',\sigma')$ with smallest eigenvalue at least $\lfloor\lambda\rfloor$, and the lattice $\Lambda(G,\sigma)$ can be embedded in a larger lattice $\Lambda(G',\sigma')$. If $\Lambda(G',\sigma')$ has the same dimension as $\Lambda(G,\sigma)$, then $(G,\sigma)$ is not maximal.\par

First, we look at maximal (non-extendable) signed graphs with smallest eigenvalue at least $-2$. In 2018, Belardo, Cioab\u{a}, Koolen and Wang \cite{belardo2018open} observed the following result.\par 

\begin{theorem}[{\cite[Theorem $3.13$]{belardo2018open}}]\label{BCKW result on -2}
If $(G,\sigma)$ is a connected signed graph whose smallest eigenvalue lies inside $[-2,-1)$, then the lattice $\Lambda(G,\sigma)$ associated with $(G,\sigma)$ is one of the root lattices $A_n,D_n,E_6,E_7$ and $E_8$.
\end{theorem}

\begin{remark}
    Let $(G,\sigma)$ be a connected signed graph with smallest eigenvalue $\lambda$. 
    \begin{enumerate}[label=\rm{(\roman*)}]
        \item If $\lambda=-2$, then $\Lambda(G,\sigma)$ is generated by a set of roots, that is, squared norm $2$ vectors, and thus $\Lambda(G,\sigma)$ is even and irreducible.
        \item If $\lambda<-2$, the connectedness of $(G,\sigma)$ does not guarantee that $\Lambda(G,\sigma)$ is irreducible; indeed, $\Lambda(G,\sigma)$ may even contain a root or a unit vector.
    \end{enumerate}
\end{remark}
The following gives a way to construct signed graphs from a given integral lattice. Let $\Lambda$ be an integral lattice with minimum squared norm $m\geq 2$, that is, $m=\min\{(\mathbf{x},\mathbf{x})\mid \mathbf{x}\in \Lambda,\mathbf{x}\neq\mathbf{0}\}$. A vector with norm $m$ is called a \emph{minimal vector} of $\Lambda$, and the set of minimal vectors is denoted by $\mathrm{Min}(\Lambda)$. If $S$ is a subset of $\rm{Min}(\Lambda)$ such that $(\mathbf{x},\mathbf{y})\in\{0,\pm1\}$ for any distinct pair $\mathbf{x}$ and $\mathbf{y}$ in $S$, then one may define a signed graph $(G,\sigma)_S$ as follows. The vertex set of $G$ is $S$, and two distinct vertices $\mathbf{x},\mathbf{y}\in S$ are adjacent whenever $(\mathbf{x},\mathbf{y})\neq 0$. For each edge $\{\mathbf{x},\mathbf{y}\}\in E(G)$, define
\[
\sigma(\{\mathbf{x},\mathbf{y}\})=
\begin{cases}
+,& \text{if }(\mathbf{x},\mathbf{y})=1,\\[2mm]
-,& \text{if }(\mathbf{x},\mathbf{y})=-1.
\end{cases}
\]\par

Let $\Lambda$ be an irreducible root lattice, that is, $\Lambda$ is one of $A_n,D_n,E_6,E_7$ and $E_8$, and let $S$ be a subset of $\mathrm{Min}(\Lambda)$ consisting of one representative from each pair $\{\mathbf{x},-\mathbf{x}\}$. The resulting signed graph $(G,\sigma)_S$ is connected with smallest eigenvalue $-2$. In addition,
\begin{enumerate}[label=\rm{(\roman*)}]
    \item  If $\Lambda$ is $A_n$, then the signed graph $(G,\sigma)_S$ has order $\frac{n(n+1)}{2}$ and valency $2n-2$.
    \item  If $\Lambda$ is $D_n$, then the signed graph $(G,\sigma)_S$ has order $n(n-1)$ and valency $4n-8$.
    \item If $\Lambda$ is $E_6$, then the signed graph $(G,\sigma)_S$ has order $36$ and valency $20$; if $\Lambda$ is $E_7$, then the signed graph $(G,\sigma)_S$ has order $63$ and valency $32$; and if $\Lambda$ is $E_8$, then the signed graph $(G,\sigma)_S$ has order $120$ and valency $56$.
\end{enumerate}
Moreover, $(G,\sigma)_S$ is maximal unless $\Lambda=A_7, A_8$ or $D_8$, and  $(G,\sigma)_S$ is non-extendable if and only if $\Lambda=E_8$. Indeed, there is exactly one switching class of non-extendable signed graph with smallest eigenvalue at least $-2$. \par

The discussion now turns to signed graphs with smallest eigenvalue in $[-3,-2)$. We extend 
Theorem \ref{main result} as follows.\par

\begin{theorem}\label{main result 2}
    There exists a positive integer $\kappa_4$ such that, if a connected signed graph $(G,\sigma)$ has smallest eigenvalue at least $-3$ and minimum valency at least $\kappa_4$, then the lattice $\Lambda(G,\sigma)$ is a sublattice of a direct sum of the standard lattice $\mathbb{Z}^n$ and copies of the root lattice $E_8$. 
Furthermore, there exist infinitely many connected signed graphs with smallest eigenvalue at least $-3$ containing $(G,\sigma)$ as a proper induced subgraph, and thus $(G,\sigma)$ is extendable.
\end{theorem}

Later in this subsection, an example of non-extendable regular signed graph with order $2300$, valency $891$ and smallest eigenvalue $-3$ is presented. This example shows that the constant $\kappa_4$ in Theorem \ref{main result 2} is at least $892$. \par

We now proceed to construct certain non-extendable signed graph with smallest eigenvalue at least $-3$. To construct such signed graphs, the lattices generated by squared norm $3$ vectors without roots will play a central role. As an example, given an irreducible unimodular lattice $\Lambda$ with minimum squared norm $3$ generated by $\rm{Min}(\Lambda)$, choose $S$ as a subset of $\mathrm{Min}(\Lambda)$ consisting of one representative from each pair $\{\mathbf{x},-\mathbf{x}\}$. One can verify that $(\mathbf{x},\mathbf{y})\in\{0,\pm1\}$ for any distinct pair $\mathbf{x}$ and $\mathbf{y}$ in $S$  and the resulting signed graph $(G,\sigma)_S$ is connected with smallest eigenvalue $-3$. Furthermore, since $\Lambda$ is irreducible, unimodular and has minimum squared norm $3$, $(G,\sigma)_S$ is non-extendable and has smallest eigenvalue $-3$ (see Lemma \ref{lattice to non-extendable of sg}).\par

Many irreducible unimodular lattices generated by squared norm $3$ vectors without roots are known. The table below lists all irreducible unimodular lattices without roots up to dimension $29$. 

{\renewcommand{\arraystretch}{1.3}
\begin{table}[H]
    \centering
    \caption{Irreducible unimodular lattices without roots up to dimension $29$}
    \label{tab_lattice}

    \begin{tabular}{c|c|c|c}
        \hline
        Dimension& Minimum squared norm & Kissing number & Number of lattices  \\
        \hline
        $23$ & $3$ & $4600$ & $1$ (the shorter Leech lattice) \\ \hline
        \multirow{2}{*}{$24$} & $3$ & $4096$ & $1$ (the odd Leech lattice )\\ \cline{2-4}
         & $4$  & $196560$ & $1$ (the Leech lattice)\\ \hline
        \multirow{2}{*}{$26$} & \multirow{2}{*}{$3$} & \multirow{2}{*}{$3120$} & $1$ (the Conway-Borcherds  \\ 
        & & & unimodular lattice)\\ \hline
        \multirow{2}{*}{$27$} & $3$ & $2664$ & $2$ \\ \cline{2-4}
        & $3$  & $1640$ & $1$ \\ \hline
       \multirow{2}{*}{$28$}& $3$  & $2240$ & $36$   \\\cline{2-4}
        & $3$  & $1728$ & $2$  \\ \hline
        \multirow{2}{*}{$29$}& $3$  & $1856$ & $9987$  \\\cline{2-4}
        & $3$  & $1600$ & $105$ \\
        \hline
    \end{tabular}
\end{table}
}

\begin{remark}
    \begin{enumerate}[label=\rm{(\roman*)}]
        \item All lattices in Table \ref{tab_lattice}, except the Leech lattice, have minimum squared norm $3$ and are generated by its minimal vectors. For a list of Gram matrices of bases of these lattices, up to dimension $28$, chosen from vectors of minimum squared norm $3$, see \cite{MBperfectlattice}.
        \item Borcherds \cite{borcherdsthesis} classified all irreducible unimodular lattices up to dimension $25$ and irreducible unimodular lattices without roots of dimension $26$. Among them, only three have minimum squared norm $3$ and are generated by its minimal vectors, namely the shorter Leech lattice, the odd Leech lattice, and the Conway-Borcherds unimodular lattice, of dimension $23, 24$, and $26$, respectively. For further details on these three lattices, we refer to \cite{borcherdsthesis,Conway1999Sphere,NSlattice}. 
        \item Bacher and Venkov \cite{Bacher2001reseaux} classified the irreducible unimodular lattices without roots of dimension $27$ and $28$. 
        \item In \cite{AllombertChenevier2025}, Allombert and Chenevier classified the irreducible unimodular lattices without roots of dimension $29$.
        \item There are many irreducible unimodular lattices in dimension $30$ to $33$ with minimum squared norm $3$ \cite{conway1998a,king2003a}, but it is unknown whether any of them is generated by their squared norm $3$ vectors.
    \end{enumerate}
\end{remark}

In what follows, attention will be restricted to the signed graphs constructed from the lattices listed in Table \ref{tab_lattice}, excluding the Leech lattice.  For such a lattice $\Lambda$, let $S$ be a subset of $\mathrm{Min}(\Lambda)$ consisting of exactly one representative from each pair $\{\mathbf{x},-\mathbf{x}\}$. The resulting signed graph $(G,\sigma)_S$ has smallest eigenvalue at least $-3$ and its order is half the number of minimal vectors of $\Lambda$, namely half the kissing number of $\Lambda$.
\begin{enumerate}[label=\rm{(\roman*)}]
    \item  If $\Lambda$ is the shorter Leech Lattice, then the signed graph $(G,\sigma)_S$ has a strongly regular graph with parameter $(2300,891,378,324)$ as its underlying graph.
    \item  If $\Lambda$ is the odd Leech Lattice, then the signed graph $(G,\sigma)_S$ has a strongly regular graph with parameter $(2048,759,310,264)$ as its underlying graph.
    \item If $\Lambda$ is the Conway-Borcherds unimodular lattice, then the signed graph $(G,\sigma)_S$ has order $1560$. If $\Lambda$ is one of the three lattices of dimension $27$, then the signed graph $(G,\sigma)_S$ has order $1332$ or $820$; if $\Lambda$ is one of the $38$ lattices of dimension $28$, then the signed graph $(G,\sigma)_S$ has order $1120$ or $864$; and if $\Lambda$ is one of the $10092$ lattices of dimension $29$, then the signed graph $(G,\sigma)_S$ has order $928$ or $800$.
\end{enumerate}
It is worth mentioning that all these signed graphs are non-extendable.\par

The preceding constructions show that there are many (switching classes of) non-extendable signed graphs with smallest eigenvalue $-3$. This naturally leads to the following problem.
\begin{problem}\label{pro non-extendable of sg}
    Are there infinitely many switching classes of non-extendable signed graphs with smallest eigenvalue at least $-3$?
\end{problem}

In light of Theorem \ref{main result 2}, we expect the answer to Problem \ref{pro non-extendable of sg} to be negative. If this is the case, then it is interesting to classify non-extendable signed graphs with smallest eigenvalue $-3$.\par

Observe that, for a maximal signed graph $(G,\sigma)$ with smallest eigenvalue $-2$, if the dimension of $\Lambda(G,\sigma)$ tends to infinity, then the valency of $(G,\sigma)$ is $2r-2$ or $4r-8$ and also tends to infinity, as $\Lambda(G,\sigma)$ is $A_r$ or $D_r$. We wonder whether an analogous statement holds for maximal signed graphs with smallest eigenvalue $-3$.\par

Our study of signed graphs with fixed smallest eigenvalue is also motivated by considerations arising from certain biangular line systems. A line system $\mathcal{L}$ in the Euclidean space $\mathbb{R}^d$, is called a system of \emph{equiangular lines} if the angle between each pair of lines is the same, and a system of \emph{biangular lines} if the angle between each pair of lines has only two possibilities. It is of particular interest to investigate systems of biangular lines with angles $\arccos \frac{1}{\alpha}$ and $90\degree$, where $\alpha$ is a positive real number. A natural and important problem, posed by Koolen et. al. \cite{koolen2026a}, is to determine the maximum number of biangular lines in $\mathbb{R}^d$ with angles $\arccos\frac{1}{\alpha}$ and $90\degree$. \par

Let $\mathcal{L}$ be a system of $n$ biangular lines in $\mathbb{R}^d$ with angles $\arccos\frac{1}{\alpha}$ and $90\degree$. Represent the lines in $\mathcal{L}$ by unit vectors $\mathbf{u}_1,\ldots\mathbf{u}_n$ and denote by $Gr$ the Gram matrix of these vectors. With this notation, the matrix $\alpha(Gr-\mathbf{I})$ is a symmetric matrix with $0$ on the diagonal and $0$ or $\pm1$ off the diagonal, and can therefore be regarded as the adjacency matrix of a signed graph with smallest eigenvalue at least $-\alpha$, see also \cite{Balla2025equiangular,Koolen2021recent,koolen2026a} for a similar discussion. This connection highlights the significance of studying of signed graphs with fixed smallest eigenvalue in investigating the maximum number of such biangular lines.\par

This paper proceeds as follows. Section \ref{sec def and pre} contains the necessary definitions, notation and properties of lattices and Hoffman signed graphs, and Section \ref{sec of proof} is devoted to the proof of Theorem \ref{main result} and Theorem \ref{main result 2}.

\section{Definitions and preliminaries}\label{sec def and pre}

This section presents the basic definitions, notation, and results that will be used  in the sequel.

\subsection{Lattices}\label{subsec lattices}

Briefly, the lattice-theoretic background needed in this paper is reviewed in this subsection. For more details, see \cite{Conway1999Sphere} and \cite{Ebeling2013lattices}.\par

Let $\mathbb{R}^n$ be an $n$-dimensional Euclidean space, equipped with the canonical \emph{inner product} $(\mathbf{x},\mathbf{y}):=\mathbf{x}^T\mathbf{y}$, where $\mathbf{x}^T$ denotes the transpose of $\mathbf{x}$. The number $||\mathbf{x}||^2=(\mathbf{x},\mathbf{x})$ will be called the \emph{squared norm} of $\mathbf{x}\in\mathbb{R}^n$. A \emph{lattice} $\Lambda$ in $\mathbb{R}^n$ is a discrete additive subgroup of $\mathbb{R}^n$. If the inner product of any two vectors in a lattice $\Lambda$ is integral, then $\Lambda$ is said to be an \emph{integral lattice}. An integral lattice is called \emph{even} if the squared norm of every vector is even.\par 

The \emph{dimension} of a lattice $\Lambda$ is the dimension of the real vector space spanned by $\Lambda$. For an $n$-dimensional integral lattice $\Lambda$, its \emph{dual lattice} is $\Lambda^*:=\{\mathbf{x}\in \mathbb{R}^n\mid (\mathbf{x},\mathbf{y})\in \mathbb{Z}$ for each $y\in\Lambda\}$, and $\Lambda$ is called \emph{unimodular} if $\Lambda=\Lambda^*$. \par

For two orthogonal lattices $\Lambda_1$ and $\Lambda_2$, that is, $(\mathbf{x},\mathbf{y})=0$ for all $\mathbf{x}\in\Lambda_1$, $\mathbf{y}\in\Lambda_2$, their \emph{direct sum}, denoted by $\Lambda_1\oplus\Lambda_2$, is the lattice $\{\mathbf{x}+\mathbf{y}\mid \mathbf{x}\in\Lambda_1,\mathbf{y}\in\Lambda_2\}$. 
A lattice is called \emph{irreducible} if it is not a non-trivial direct sum of orthogonal sublattices; otherwise, it is \emph{reducible}.\par

If $X$ is a finite set of vectors in  $\mathbb{R}^n$ such that $(\mathbf{x},\mathbf{y})\in \mathbb{Z}$ for all $\mathbf{x,y}\in X$, then $\Lambda:=\left\{ \sum_{\mathbf{x}\in X}\alpha_{\mathbf{x}}\mathbf{x}\mid \alpha_{\mathbf{x}}\in\mathbb{Z}\right\}$ is an integral lattice. In this case, $\Lambda$ is said to be \emph{generated} by $X$. An integral lattice generated by a set of vectors of squared norm $2$ is called a \emph{root lattice}.

\begin{theorem}[{\cite{Witt1941Spiegelungsgruppen}}]\label{irreducible root lattices}
    Each irreducible root lattice is one of the following:
    \begin{enumerate}[label=\rm{(\roman*)}]
        \item $A_n=\{\mathbf{x}\in\mathbb{Z}^{n+1}\mid \sum\limits_{i=1}^{n+1} x_i=0\}$ for $n\geq1$,
        \item $D_n=\{\mathbf{x}\in\mathbb{Z}^n\mid \sum\limits_{i=1}^n x_i$ is even$\}$ for $n\geq4$,
        \item $E_8=D_8\cup (\mathbf{c}+D_8)$, where $\mathbf{c}=\frac{1}{2}(1,1,1,1,1,1,1,1)^T$,
        \item $E_7=\{\mathbf{x}\in E_8\mid \sum\limits_{i=1}^8 x_i=0\}$,
        \item $E_6=\{\mathbf{x}\in E_7\mid x_7+x_8=0\}$.
    \end{enumerate}
\end{theorem}

 Note that the root lattices $A_n$ and $D_n$ can be realized as sublattices of a standard lattice, while $E_6$ and $E_7$ occur as sublattices of  $E_8$.

\begin{lemma}\label{lattice to non-extendable of sg}
Given a signed graph $(G,\sigma)$ with smallest eigenvalue $\lambda$, where $\lambda\leq-1$ is a real number.
\begin{enumerate}[label=\rm{(\roman*)}]
    \item If $\Lambda(G,\sigma)$ is irreducible, then $(G,\sigma)$ is connected.
    \item If $\Lambda(G,\sigma)$ is irreducible and unimodular with minimum squared norm $-\lfloor\lambda\rfloor$, then $(G,\sigma)$ is maximal implies that $(G,\sigma)$ is non-extendable.
\end{enumerate}

\end{lemma}
\begin{proof}
  $\rm{(i)}$ Suppose that $(G,\sigma)$ contains at least two connected components, then there exist two matrices $N_1$ and $N_2$ such that 
    $$A(G,\sigma)-\lfloor\lambda\rfloor\mathbf{I}=\begin{pmatrix} N_1^TN_1 & O \\ O & N_2^TN_2 \end{pmatrix}, $$
where $O$ is the all-zeros matrix. Let $\Lambda_1$ and $\Lambda_2$ denote the lattices generated by the columns of $N_1$ and $N_2$, respectively, and then $\Lambda(G,\sigma)=\Lambda_1\oplus \Lambda_2$, which contradicts the condition that $\Lambda(G,\sigma)$ is irreducible.\par

  $\rm{(ii)}$ Suppose that there exists a connected signed graph $(G',\sigma')$ with smallest eigenvalue at least $\lfloor\lambda\rfloor$ which contains $(G,\sigma)$ as an proper induced  subgraph. Assume that 
  
  $$A(G,\sigma)-\lfloor\lambda\rfloor\mathbf{I}=N^TN\textrm{ and } A(G',\sigma')-\lfloor\lambda\rfloor\mathbf{I}=(N')^TN',$$ 
  where the columns of $N$ and $N'$ are indexed by $V(G)$ and $V(G')$, respectively. \par
  
  Let $y\in V(G')\setminus V(G)$.   Since $(G,\sigma)$ is maximal, one has $N'_y\notin\Lambda(G,\sigma)$.  It is claimed that $N'_y\perp N_x$ for each vertex $x\in V(G)$. Let $\mathbf{v}_y$ be the projection of $N'_y$ onto the vector space spanned by the columns of $N$. If follows that  $(\mathbf{v}_y, N_x)=(N'_y,N_x)$ is an integer for each vertex $x\in V(G)$, and thus $\mathbf{v}_y\in\Lambda(G,\sigma)$, as $\Lambda(G,\sigma)$ is unimodular. Since $(N'_y,N'_y)=-\lfloor\lambda\rfloor$ equals the minimum squared norm of $\Lambda(G,\sigma)$, either $\mathbf{v}_y=N'_y$ or $\mathbf{v}_y=\mathbf{0}$ holds. The first possibility is impossible, because $N'_y\notin\Lambda(G,\sigma)$. Hence $\mathbf{v}_y=\mathbf{0}$  and $N'_y\perp N_x$ for each vertex $x\in V(G)$. This implies that $(G,\sigma)$ is a connected component of $(G',\sigma)$, which contradicts the connectedness of $(G',\sigma')$. \end{proof}

\subsection{Hoffman signed graphs}

A \emph{Hoffman signed graph} $\mathfrak{h}=(H,\sigma,\ell)$ consists of a signed graph $(H,\sigma)$ together with a labeling map $\ell:V(H)\rightarrow\{f,s\}$, subject to the following conditions: any pair of vertices with label $f$ are non-adjacent and any vertex with label $f$ has at least one neighbor. A vertex labeled $s$ is called a \emph{slim vertex}, while a vertex labeled $f$ is called a \emph{fat vertex}. The set of slim vertex and fat vertex of $\mathfrak{h}$ are denoted by $V_s(\mathfrak{h})$ and $V_f(\mathfrak{h})$, respectively. A Hoffman signed graph is called \emph{fat}, if each slim vertex of the  $\mathfrak{h}$ has at least one fat neighbor. The \emph{slim graph} of $\mathfrak{h}$, denoted by $(H_s,\sigma_s)$, is the subgraph $(H,\sigma)$ induced on $V_s(\mathfrak{h})$. A signed graph may be regarded as a Hoffman signed graph all of whose vertices are slim. Conversely, a Hoffman signed graph $\mathfrak{h}=(H,\sigma,\ell)$ with only slim vertices will be identified with the signed graph $(H,\sigma)$.\par
 
Given two Hoffman signed graphs $\mathfrak{h}_1=(H_1,\sigma_1,\ell_1)$ and $\mathfrak{h}=(H,\sigma,\ell)$,  $\mathfrak{h}_1$ is called a Hoffman \emph{subgraph} of $\mathfrak{h}$, if $(H_1,\sigma_1)$ is a subgraph of $(H,\sigma)$ and $\ell_1=\ell|_{V(H_1)}$, and $\mathfrak{h}_1$ is called an \emph{induced Hoffman subgraph} of $\mathfrak{h}$, if $(H_1,\sigma_1)$ is an induced subgraph of $(H,\sigma)$ and $\ell_1=\ell|_{V(H_1)}$.
  Two Hoffman signed graphs $\mathfrak{h}=(H,\sigma,\ell)$ and $\mathfrak{h'}=(H',\sigma',\ell')$ are \emph{isomorphic}, if there exists an isomorphism between $(H,\sigma)$ and $(H',\sigma')$ that preserves both edge signs and vertex labels.\par
  
  For a Hoffman signed graph $\mathfrak{h}=(H,\sigma,\ell)$, let
  $$A=\begin{pmatrix}
    A_s&C\\
    C^T&O
  \end{pmatrix}$$
  be the adjacency matrix of $(H,\sigma)$, where $A_s$ is the adjacency matrix of the slim graph of $(H,\sigma)$.
  The matrix $S(\mathfrak{h}):=A_s-CC^T$ is called the \emph{special matrix} of $\mathfrak{h}$. The \emph{eigenvalues} of $\mathfrak{h}$ are defined to be the eigenvalues of $S(\mathfrak{h})$, and the smallest eigenvalue of $\mathfrak{h}$ is denote by $\lambda_{\min}(\mathfrak{h})$. In particular, if $\mathfrak{h}=(H,\sigma,\ell)$ has no fat vertices, then $S(\mathfrak{h})=A(H,\sigma)$, and thus the eigenvalues of $\mathfrak{h}$ are precisely the eigenvalues of $(H,\sigma)$. \par
  
%Let $\mathfrak{h}=(H,\sigma,\ell)$ be a Hoffman signed graph and $U$ be a subset of $V(H)$. A Hoffman signed graph $\mathfrak{h'}=(H',\sigma',\ell')$ is said to be obtained from $\mathfrak{h}$ by \emph{switching} with respect to $U$, if $(H',\sigma')$ is obtained from $(H,\sigma)$ by switching with respect to $U$ and $\ell'(x)=\ell(x)$ for every $x\in V(H')$. Two Hoffman signed graphs are called \emph{switching equivalent}, if one can be obtained from the other by switching. \par

Let $\mathfrak{h}$ be a Hoffman signed graph and $n$ be a positive integer. A positive clique with $n$ vertices, denote by $(K_n,+)$, is the signed graph whose underlying graph is complete and with each edge signed $+$. Define $G(\mathfrak{h},n)$ as the signed graph obtained from $\mathfrak{h}$ by replacing each fat vertex with a positive clique $(K_n,+)$. The following result originates from Hoffman \cite{hoffman1977onsigned}, and was later reproved by Gavrilyuk et al. \cite{Gavrilyuk2021signed} in the framework of Hoffman signed graphs.

\begin{theorem}[{\cite[Theorem $4.2$]{Gavrilyuk2021signed}}]\label{Hoffman-Ostrowski Theorem of signed case}
 For every Hoffman signed graph $\mathfrak{h}$ and every positive integer $n$, $$\lambda_{\min}(G(\mathfrak{h},n))\geq\lambda_{\min}(\mathfrak{h}), \textrm{ and }\lim\limits_{n\rightarrow+\infty}\lambda_{\min}(G(\mathfrak{h},n))=\lambda_{\min}(\mathfrak{h}).$$
\end{theorem}

\begin{theorem}[{\cite[Theorem $1.1$, Proposition $4.12$ and Theorem $5.2$]{koolen2026a}}]\label{existence of fat HSG}
Let $\lambda< -1$ be a real number and $v_f,v_s,p$ be non-negative integers. There exist positive integers $m_\lambda$ and $q(\lambda,v_f,v_s,p)$ with the following property.\par
For every integer $n\geq q(\lambda,v_f,v_s,p)$, there exists a positive integer $d(\lambda,n)$ such that, for any signed graph $(G,\sigma)$ with smallest eigenvalue at least $\lambda$ and minimum valency at least $d(\lambda,n)$, and for any Hoffman signed graph $\mathfrak{h}$ with at most $v_f$ fat vertices and at most $v_s$ slim vertices, there exists a fat Hoffman signed graph $\mathfrak{g}(G,\sigma,m_\lambda,n)$ such that the following hold.
\begin{enumerate}[label=\rm{(\roman*)}]
\item The signed graph $(G,\sigma)$ is the slim graph of $\mathfrak{g}(G,\sigma,m_\lambda,n)$.
\item If $\mathfrak{h}$ is an induced Hoffman subgraph of $\mathfrak{g}(G,\sigma,m_\lambda,n)$, then the signed graph $G(\mathfrak{h},p)$ is an induced subgraph of $(G,\sigma)$.

\end{enumerate}
\end{theorem}

This subsection concludes by introducing the representation and the reduced representation of Hoffman signed graphs.\par

   For a Hoffman signed graph $\mathfrak{h}$, a mapping 
 $\phi:V(\mathfrak{h})\rightarrow\mathbb{R}^n$ satisfying
 $$(\phi(x),\phi(y))=\begin{cases}
    m, & \text{if } x=y\in V_s(\mathfrak{h}), \\[2mm]
    1, & \text{if } x=y\in V_f(\mathfrak{h}), \\[2mm]
    1, & \text{if } x,y \text{ are positively adjacent},\\[2mm]
    -1, & \text{if } x,y \text{ are negatively adjacent},\\[2mm]
    0, & \text{otherwise},
  \end{cases}$$ 
is called a \emph{representation of $\mathfrak{h}$ of squared norm $m$}. The lattice generated by $\{\phi(x)\mid x\in V(\mathfrak{h})\}$ is denoted by $\Lambda(\mathfrak{h},m)$. Note that the isomorphism class of $\Lambda(\mathfrak{h},m)$ depends only on $\mathfrak{h}$ and $m$, and is independent of $\phi$, justifying the notation.\par
 For a Hoffman signed graph $\mathfrak{h}=(H,\sigma,\ell)$, define $$n_{\mathfrak{h}}^f(x,y):=\sum_{\substack{z\in V_f(\mathfrak{h})\\\{x,z\},\{y,z\}\in E(\mathfrak{h}) }}\big|\{z\mid \sigma(\{x,z\})=\sigma(\{y,z\})\}\big|-\sum_{\substack{z\in V_f(\mathfrak{h})\\ \{x,z\},\{y,z\}\in E(\mathfrak{h})}}\big|\{z\mid \sigma(\{x,z\})\neq\sigma(\{y,z\})\}\big|$$ for $x,y\in V_s(\mathfrak{h})$.\par

 For a Hoffman signed graph $\mathfrak{h}$, a mapping $\psi:V_s(\mathfrak{h})\rightarrow\mathbb{R}^n$ satisfying
 $$(\psi(x),\psi(y))=\begin{cases}
    m-n_{\mathfrak{h}}^f(x,y), & \text{if } x=y, \\[2mm]
    1-n_{\mathfrak{h}}^f(x,y), & \text{if } x,y \text{ are positively adjacent},\\[2mm]
    -1-n_{\mathfrak{h}}^f(x,y), & \text{if } x,y \text{ are negatively adjacent},\\[2mm]
    -n_{\mathfrak{h}}^f(x,y), & \text{otherwise},
  \end{cases}$$ is called a \emph{reduced representation of $\mathfrak{h}$ of squared norm $m$}. The lattice generated by $\{\psi(x)\mid x\in V_s(\mathfrak{h})\}$ is denoted by $\Lambda^{\rm{red}}(\mathfrak{h},m)$. Note that the isomorphism class of $\Lambda^{\rm{red}}(\mathfrak{h},m)$ also depends only on $\mathfrak{h}$ and $m$, and is independent of $\psi$, justifying the notation.

\begin{lemma}[{\cite[Lemma $4.5$ and Theorem $4.6$]{koolen2026a}}]
\label{reduced representation lemma}
Let $\lambda\leq -1$ be an integer. If a Hoffman signed graph $\mathfrak{h}$ has smallest eigenvalue at least $\lambda$, then $$\Lambda(\mathfrak{h},|\lambda|)=\Lambda^{\rm{red}}(\mathfrak{h},|\lambda|)\oplus\mathbb{Z}^{|V_f(\mathfrak{h})|}.$$ 
\end{lemma}

\section{Proofs of Theorem \ref{main result} and Theorem \ref{main result 2}}\label{sec of proof}

This section presents the proof of Theorem \ref{main result} and Theorem \ref{main result 2}. \par

Let $\lambda\leq-1$ be a real number. A fat Hoffman signed graph is called a \emph{minimal forbidden fat Hoffman signed graph} $\mathfrak{f}$ for $\lambda$ if its smallest eigenvalue is less than $\lambda$, while every proper fat induced Hoffman subgraph of $\mathfrak{f}$ has smallest eigenvalue at least $\lambda$. Let $\mathcal{F}(\lambda)$ denote the set of pairwise non-isomorphic minimal forbidden fat Hoffman signed graphs for $\lambda$.\par

The first step is to prove the set $\mathcal{F}(-3)$ is finite. To prepare for the proof, we recall a result due to Koolen, Yang and Yang \cite{koolen2018on}.

\begin{theorem}[{\cite[Theorem $4.2$]{koolen2018on}}]\label{kyy: min forbidden finite for -2}
    Let $M$ be a real symmetric matrix, whose diagonal entries are $0$ or $-1$ and off-diagonal entries are integers. If the smallest eigenvalue of M is less than $-2$, then one of its principal submatrices of order at most $10$ also has an eigenvalue less than $-2$. 
\end{theorem}

\begin{theorem}\label{min forbidden finite for -2}
Let $\mathfrak{f}\in\mathcal{F}(-3)$ and $S(\mathfrak{f})$ be the special matrix of $\mathfrak{f}$.
\begin{enumerate}[label=\rm{(\roman*)}]
    \item If $\mathfrak{f}$ contains a slim vertex with at least three fat neighbors, then  $S(\mathfrak{f})$ is
$$(-4)\text{ or }\begin{pmatrix} -3 & a_2 \\ a_2 & a_1 \end{pmatrix} , $$
where $a_1,a_2$ are integers such that $a_1\in\{-1,-2,-3\}$ and $1\leq|a_2|\leq 1-a_1$.
\item If each slim vertex of $\mathfrak{f}$ has at most two fat neighbors, then $\mathfrak{f}$ has at most $10$ slim vertices.
\end{enumerate}
In particular, the set $\mathcal{F}(-3)$ is finite.
\end{theorem}
\begin{proof}
Let $\mathfrak{f}\in\mathcal{F}(-3)$ and consider its special matrix $S(\mathfrak{f})$. By the definition of the set $\mathcal{F}(-3)$,  $\lambda_{\min}(S(\mathfrak{f}))<-3$ and every principal submatrix of $S(\mathfrak{f})$ has smallest eigenvalue at least $-3$. One may verify that, for any slim vertex $x$ of $\mathfrak{f}$
$$S(\mathfrak{f})_{xx}=-|\{F\mid F\textrm{ is a fat neighbor of }x\}|,$$
and for any pair of distinct slim vertices $x,y$ of $\mathfrak{f}$ $$|S(\mathfrak{f})_{xy}|\leq 1+\min\{|S(\mathfrak{f})_{xx}|,|S(\mathfrak{f})_{yy}|\}.$$
\par
$\rm{(i)}$ Assume that $\mathfrak{f}$ contains a slim vertex $x$ with at least three fat neighbors, then the special matrix $S(\mathfrak{f})$ is $$(-4)\text{ or }\begin{pmatrix} -3 & a_2 \\ a_2 & a_1 \end{pmatrix} , $$
where $a_1,a_2$ are integers such that $a_1\in\{-1,-2,-3\}$ and $1\leq|a_2|\leq 1-a_1$.\par
$\rm{(ii)}$ Since each slim vertex of $\mathfrak{f}$ has at most two fat neighbors and $\mathfrak{f}$ is fat, the diagonal entries of $S(\mathfrak{f})$ are $-1$ or $-2$. Thus, the matrix $S(\mathfrak{f})+\mathbf{I}$ is a real symmetric matrix whose diagonal entries are $0$ or $-1$ and off-diagonal entries are integers, and has smallest eigenvalue less than $-2$. As every principal submatrix of $S(\mathfrak{f})+\mathbf{I}$ has smallest eigenvalue st least $-2$, Theorem \ref{kyy: min forbidden finite for -2} implies that the order of $S(\mathfrak{f})$ is at most $10$. Hence, $\mathfrak{f}$ has at most $10$ slim vertices. \par

Therefore, for any $\mathfrak{f}\in \mathcal{F}(-3)$, $\mathfrak{f}$ has at most $10$ slim vertices and at most $20$ fat vertices, which yields that $\mathcal{F}(-3)$ is finite. This completes the proof.
\end{proof}

\begin{theorem}\label{existence of fat HFSG for -3}
Let $\Delta=\min\{-3-\lambda_{\min}(\mathfrak{f})\mid \mathfrak{f}\in\mathcal{F}(-3)\}$. There exists a positive integer $\kappa'$ such that, if a connected signed graph $(G,\sigma)$ has smallest eigenvalue $\lambda_{\min}(G,\sigma)> -3-\Delta$ and minimum valency at least $\kappa'$, then $(G,\sigma)$ is the slim graph of a fat Hoffman signed graph with smallest eigenvalue at least $-3$. 
\end{theorem}
\begin{proof}
For each $\mathfrak{f}\in \mathcal{F}(-3)$, there exists a positive integer $p_{\mathfrak{f}}$ such that $\lambda_{\min}(G(\mathfrak{f},p_{\mathfrak{f}}))\leq -3-\Delta$, by Theorem \ref{Hoffman-Ostrowski Theorem of signed case}. Since $\mathcal{F}(-3)$ is finite, we may denote $p=\max\{p_{\mathfrak{f}}\mid \mathfrak{f}\in\mathcal{F}(-3)\}$. 
Let $\lambda=-3-\Delta$, $v_f=20$, $v_s=10$ and let $m_{\lambda}$, $q(\lambda,v_f,v_s,p)$ and $d(\lambda,q(\lambda,v_f,v_s,p))$ be integers such that Theorem \ref{existence of fat HSG} holds. Let $\kappa':=d(\lambda,q(\lambda,v_f,v_s,p))$.\par
For a connected signed graph $(G,\sigma)$ with smallest eigenvalue $\lambda_{\min}(G,\sigma)>-3-\Delta$, if the minimum valency of $(G,\sigma)$ is at least $\kappa'$, then by Theorem \ref{existence of fat HSG}, there exists a fat Hoffman signed graph $\mathfrak{g}=\mathfrak{g}(G,\sigma,m_{\lambda},q(\lambda,v_f,v_s,p))$ such that $(G,\sigma)$ is the slim graph of $\mathfrak{g}$, and if $\mathfrak{h}$ is an induced Hoffman subgraph of $\mathfrak{g}$ with at most $v_f$ fat vertices and at most $v_s$ slim vertices, then the signed graph $G(\mathfrak{h},p)$ is an induced subgraph of $(G,\sigma)$.\par

 It remains to show that $\lambda_{\min}(\mathfrak{g})\geq -3$. Suppose, for a contradiction, $\lambda_{\min}(\mathfrak{g})<-3$, then $\mathfrak{g}$ contains a Hoffman signed graph $\mathfrak{f}\in\mathcal{F}(-3)$ as an induced Hoffman subgraph. Hence, the signed graph $G(\mathfrak{f},p)$ is an induced subgraph of $(G,\sigma)$ and $-3-\Delta\geq\lambda_{\min}(G(\mathfrak{f},p))\geq\lambda_{\min}(G,\sigma)>-3-\Delta$. This gives a contradiction.\par
This completes the proof.
\end{proof}

Now we prove Theorem \ref{main result}.

\noindent
\textbf{Proof of Theorem \ref{main result}}. 
Let $\varepsilon:=\Delta$ be a real number and $\kappa_3:=\kappa'$ be a positive integer such that Theorem \ref{existence of fat HFSG for -3} holds. Let  $(G,\sigma)$ be a connected signed graph with smallest eigenvalue $\lambda_{\min}(G,\sigma)> -3-\varepsilon$ and minimum valency at least $\kappa_3$. By Theorem \ref{existence of fat HFSG for -3}, there exists a fat Hoffman signed graph $\mathfrak{g}$ with smallest eigenvalue $\lambda_{\min}(\mathfrak{g})\geq -3$ which has $(G,\sigma)$ as its slim graph, and then $\lambda_{\min}(G,\sigma)\geq \lambda_{\min}(\mathfrak{g})\geq-3$.\par

It remains to show that $\Lambda(G,\sigma)$ is a sublattice of a direct sum of the standard lattice $\mathbb{Z}^n$ and copies of the root lattice $E_8$. Without loss of generality, we may assume that $-3\leq \lambda_{\min}(G,\sigma)<-2$, by Theorem \ref{irreducible root lattices}. First, it is shown that each non-trivial irreducible sublattice of $\Lambda^{\rm{red}}(\mathfrak{g},3)$ is a standard lattice or a root lattice. Since $\lambda_{\min}(\mathfrak{g})\geq -3$, the matrix $S(\mathfrak{g})+3\mathbf{I}$ is positive semidefinite. There exists a matrix $N$ such that $S(\mathfrak{g})+3\mathbf{I}=N^TN$. As $\mathfrak{g}$ is fat and $\lambda_{\min}(\mathfrak{g})\geq -3$, each diagonal entry of $S(\mathfrak{g})$ is either $-1,-2$ or $-3$, and then the squared norm of each column of $N$ is one of $2,1$ and $0$. By Theorem \ref{irreducible root lattices}, the desired result follows.\par

 Since $\Lambda(\mathfrak{g},3)=\Lambda^{\rm{red}}(\mathfrak{g},3)\oplus \mathbb{Z}^{|V_f(\mathfrak{g})|}$ by Lemma \ref{reduced representation lemma}, each non-trivial irreducible sublattice of $\Lambda^{\rm{red}}(\mathfrak{g},3)$ is also a standard lattice or a root lattice. Thus $\Lambda(\mathfrak{g},3)$ decomposes as a direct sum of irreducible sublattices of $\mathbb{Z}^n$ or $E_8$, so does the lattice $\Lambda(G,\sigma)$ which is the sublattice of  $\Lambda(\mathfrak{g},3)$.\par

Hence the theorem follows.\qed

This paper concludes with the proof of Theorem \ref{main result 2}.

\noindent
\textbf{Proof of Theorem \ref{main result 2}}. 
Let $\kappa_4:=\max\{\kappa_3,\kappa'\}$ be a positive integer such that both Theorem \ref{main result} and Theorem \ref{existence of fat HFSG for -3} hold. Let $(G,\sigma)$ be a connected signed graph with smallest eigenvalue $\lambda_{\min}(G,\sigma)\geq -3$ and minimum valency at least $\kappa_4$. By Theorem \ref{main result}, it remains only to prove the existence of connected signed graphs with smallest eigenvalue at least $-3$ containing $(G,\sigma)$ as a proper induced subgraph. \par
Theorem \ref{existence of fat HFSG for -3} shows that there exists a fat Hoffman signed graph $\mathfrak{g}$ with smallest eigenvalue $\lambda_{\min}(\mathfrak{g})\geq -3$ which has $(G,\sigma)$ as its slim graph. For an arbitrary positive integer $q$, let $G(\mathfrak{g},q)$ be the signed graph obtained from $\mathfrak{g}$ by replacing each fat vertex with a positive clique $(K_q,+)$. Observe that $G(\mathfrak{g},q)$ contains $(G,\sigma)$ as a proper induced subgraph, and $G(\mathfrak{g},q)$ is connected, as $(G,\sigma)$ is connected. Also,
\[
\lambda_{\min}(G(\mathfrak{g},q))\geq \lambda_{\min}(\mathfrak{g})\geq -3
\]
by Theorem \ref{Hoffman-Ostrowski Theorem of signed case}. \par

This completes the proof. \qed

\section*{Acknowledgements}
J.H. Koolen is partially supported by the National Natural Science Foundation of China (No. 12471335), and the 
Anhui Initiative in Quantum Information Technologies (No. AHY150000).
Q. Yang is supported by the National Natural Science Foundation of China (No. 12401460).

\section*{Author Contributions}
All authors contributed equally in this work.

\section*{Financial disclosure}

None reported.

\section*{Conflict of interest}

The authors declare no potential conflict of interests.

\bibliography{CKLY-3}

@article {AllombertChenevier2025,
    AUTHOR = {Allombert, B. and Chenevier, G.},
     TITLE = {Unimodular hunting {II}},
   JOURNAL = {Forum Math. Sigma},
  FJOURNAL = {Forum of Mathematics. Sigma},
    VOLUME = {13},
      YEAR = {2025},
     PAGES = {e136},
      ISSN = {2050-5094},
   MRCLASS = {11H55 (11E41 11F60 11H56 11Y40)},
  MRNUMBER = {4950160},
MRREVIEWER = {Stefan\ K\"uhnlein},
       DOI = {10.1017/fms.2025.10058},
       URL = {https://doi.org/10.1017/fms.2025.10058},
}

@incollection {Bacher2001reseaux,
    AUTHOR = {Bacher, R. and Venkov, B.},
     TITLE = {R\'eseaux entiers unimodulaires sans racines en dimensions 27 et 28},
 BOOKTITLE = {R\'eseaux euclidiens, designs sph\'eriques et formes
              modulaires},
    SERIES = {Monogr. Enseign. Math.},
    VOLUME = {37},
     PAGES = {212--267},
 PUBLISHER = {Enseignement Math., Geneva},
      YEAR = {2001},
      ISBN = {2-940264-02-3},
   MRCLASS = {11H06 (11F27 11H50 11H56)},
  MRNUMBER = {1878751},
MRREVIEWER = {Gabriele\ Nebe}
}

@article {Balla2025equiangular,
    AUTHOR = {Balla, I.},
     TITLE = {Equiangular lines via matrix projection},
   JOURNAL = {Adv. Math.},
  FJOURNAL = {Advances in Mathematics},
    VOLUME = {482},
      YEAR = {2025},
     PAGES = {110620},
      ISSN = {0001-8708,1090-2082},
   MRCLASS = {14C30 (05C50)},
  MRNUMBER = {4976612},
       DOI = {10.1016/j.aim.2025.110620},
       URL = {https://doi.org/10.1016/j.aim.2025.110620},
}

@article{belardo2018open,
 TITLE = {Open problems in the spectral theory of signed graphs},
   AUTHOR ={Belardo, F. and Cioab\u{a}, S.M. and Koolen, J.H. and Wang, J.F.},
 JOURNAL = {Art Discrete Appl. Math.},
  FJOURNAL = {The Art of Discrete and Applied Mathematics},
    VOLUME = {1},
      YEAR = {2018},
    NUMBER = {2},
     PAGES = {2.10},
      ISSN = {2590-9770},
   MRCLASS = {05C50 (05C22)},
  MRNUMBER = {3997096},
MRREVIEWER = {Leila\ Parsaei Majd},
       DOI = {10.26493/2590-9770.1286.d7b},
       URL = {https://doi.org/10.26493/2590-9770.1286.d7b},
}

@phdthesis{borcherdsthesis,
  title={The Leech lattice and other lattices},
  author={Borcherds, R.E.},
school = {Trinity College, University of Cambridge},
year = {1984},
type = {Ph.D. thesis},
note = {Corrected electronic version, arXiv:math/9911195, 1999}
}

@article {conway1998a,
    AUTHOR = {Conway, J.H. and Sloane, N.J.A.},
     TITLE = {A note on optimal unimodular lattices},
   JOURNAL = {J. Number Theory},
  FJOURNAL = {Journal of Number Theory},
    VOLUME = {72},
      YEAR = {1998},
    NUMBER = {2},
     PAGES = {357--362},
      ISSN = {0022-314X,1096-1658},
   MRCLASS = {11H31 (11H06)},
  MRNUMBER = {1651697},
MRREVIEWER = {Renaud\ Coulangeon},
       DOI = {10.1006/jnth.1998.2257},
       URL = {https://doi.org/10.1006/jnth.1998.2257},
}

@book{Conway1999Sphere,
  author    = {Conway, J.H. and Sloane, N.J.A.},
  title     = {Sphere Packings, Lattices and Groups},
  edition   = {3rd},
  series    = {Grundlehren der mathematischen Wissenschaften},
  volume    = {290},
  publisher = {Springer},
  address   = {New York},
  year      = {1999},
  doi       = {10.1007/978-1-4757-6568-7}
}

@book {Ebeling2013lattices,
    AUTHOR = {Ebeling, W.},
     TITLE = {Lattices and Codes},
    SERIES = {Advanced Lectures in Mathematics},
   EDITION = {3rd},
 PUBLISHER = {Springer Spektrum, Wiesbaden},
      YEAR = {2013},
       DOI = {10.1007/978-3-658-00360-9},
       URL = {https://doi.org/10.1007/978-3-658-00360-9}
}

@article{Gavrilyuk2021signed,
AUTHOR= {Gavrilyuk, A.L. and Munemasa, A. and Sano, Y. and Taniguchi, T.},
 TITLE = {Signed analogue of line graphs and their smallest eigenvalues},
   JOURNAL = {J. Graph Theory},
  FJOURNAL = {Journal of Graph Theory},
    VOLUME = {98},
      YEAR = {2021},
    NUMBER = {2},
     PAGES = {309--325},
      ISSN = {0364-9024,1097-0118},
   MRCLASS = {05C50 (05C22 05C76 15A18 15B57)},
  MRNUMBER = {4371442},
MRREVIEWER = {Yong\ Lu},
       DOI = {10.1002/jgt.22699},
       URL = {https://doi.org/10.1002/jgt.22699},
}

@article {hoffman1977onsigned,
    AUTHOR = {Hoffman, A.J.},
     TITLE = {On signed graphs and gramians},
   JOURNAL = {Geometriae Dedicata},
  FJOURNAL = {Geometriae Dedicata},
    VOLUME = {6},
      YEAR = {1977},
    NUMBER = {4},
     PAGES = {455--470},
   MRCLASS = {15A45 (05C99)},
  MRNUMBER = {463211},
MRREVIEWER = {R.\ A.\ Brualdi},
       DOI = {10.1007/BF00147783},
       URL = {https://doi.org/10.1007/BF00147783}
}

@article{king2003a,
    AUTHOR = {King, O.D.},
     TITLE = {A mass formula for unimodular lattices with no roots},
   JOURNAL = {Math. Comp.},
  FJOURNAL = {Mathematics of Computation},
    VOLUME = {72},
      YEAR = {2003},
    NUMBER = {242},
     PAGES = {839--863},
      ISSN = {0025-5718,1088-6842},
   MRCLASS = {11H55 (11E41)},
  MRNUMBER = {1954971},
MRREVIEWER = {Hidenori\ Katsurada},
       DOI = {10.1090/S0025-5718-02-01455-2},
       URL = {https://doi.org/10.1090/S0025-5718-02-01455-2}
}

@article {Koolen2021recent,
    AUTHOR = {Koolen, J.H. and Cao, M.-Y. and Yang, Q.},
     TITLE = {Recent progress on graphs with fixed smallest adjacency
              eigenvalue: a survey},
   JOURNAL = {Graphs Combin.},
  FJOURNAL = {Graphs and Combinatorics},
    VOLUME = {37},
      YEAR = {2021},
    NUMBER = {4},
     PAGES = {1139--1178},
      ISSN = {0911-0119,1435-5914},
   MRCLASS = {05C50 (05C22 05C75 05D99 05E30 11H06)},
  MRNUMBER = {4280318},
       DOI = {10.1007/s00373-021-02296-8},
       URL = {https://doi.org/10.1007/s00373-021-02296-8},
}

@article{koolen2026a,
     author={Koolen, J.H. and Liu, J.-Y. and Yang, Q. and Cao, M.-Y.},
      title={A structure theory for signed graphs with fixed smallest eigenvalue}, 
      journal={arXiv:2602.20783},
      year={2026}
}

@article{koolen2018on,
 TITLE ={On graphs with smallest eigenvalue at least $-3$ and their lattices},
AUTHOR = {Koolen, J.H. and Yang, J.Y. and Yang, Q.},
   JOURNAL = {Adv. Math.},
  FJOURNAL = {Advances in Mathematics},
    VOLUME = {338},
      YEAR = {2018},
     PAGES = {847--864},
      ISSN = {0001-8708,1090-2082},
   MRCLASS = {05C50 (11H06)},
  MRNUMBER = {3861717},
MRREVIEWER = {Mihai\ Cipu},
       DOI = {10.1016/j.aim.2018.09.004},
       URL = {https://doi.org/10.1016/j.aim.2018.09.004},
}

@misc{MBperfectlattice,
  title={A catalogue of Perfect Lattices},
  author={Martinet, J. and Batut, C.},
  howpublished={\url{https://jamartin.perso.math.cnrs.fr/Lattices/}},
  note   = {Accessed: 26 June 2026}
}

@misc{NSlattice,
  author={Nebe, G. and Sloane, N.J.A.},
  title  = {A Catalogue of Lattices},
  howpublished={\url{http://www.math.rwth-aachen.de/~Gabriele.Nebe/LATTICES/}},
  note   = {Accessed: 26 June 2026}
}

@article {Witt1941Spiegelungsgruppen,
    AUTHOR = {Witt, E.},
     TITLE = {Spiegelungsgruppen und {A}ufz\"ahlung halbeinfacher {L}iescher
              {R}inge},
   JOURNAL = {Abh. Math. Sem. Hansischen Univ.},
  FJOURNAL = {Abhandlungen aus dem Mathematischen Seminar der Hansischen
              Universit\"at},
    VOLUME = {14},
      YEAR = {1941},
     PAGES = {289--322},
      ISSN = {0025-5858},
   MRCLASS = {09.1X},
  MRNUMBER = {5099},
MRREVIEWER = {O.\ F. G. Schilling},
       DOI = {10.1007/BF02940749},
       URL = {https://doi.org/10.1007/BF02940749},
}

\end{document}